\documentclass[a4paper,11pt]{article}
\usepackage[english]{babel}
\usepackage[ansinew]{inputenc}

\usepackage{amsmath}
\usepackage{amssymb}
\usepackage{amsthm}
\usepackage[hidelinks]{hyperref}
\usepackage{enumitem}
\usepackage{bm}

\usepackage{csquotes}
\MakeOuterQuote{"}

\newtheorem{theorem}{Theorem}[section]
\newtheorem{corollary}[theorem]{Corollary}
\newtheorem{lemma}[theorem]{Lemma}

\renewcommand{\epsilon}{\varepsilon}
\newcommand{\R}{\mathbb{R}}
\newcommand{\Q}{\mathbb{Q}}
\newcommand{\N}{\mathbb{N}}

\author{Edmund \sc{Weitz}\footnote{E-mail: edmund.weitz@haw-hamburg.de.}\\
Hamburg University of Applied Sciences}
\title{Hausdorff's forgotten proof that\\ almost all numbers are normal}
\date{January 31, 2021}
\begin{document}
\maketitle

\begin{abstract}
  In 1914, Felix Hausdorff published an elegant proof that almost all numbers
  are simply normal in base~2.  We generalize this proof to show that
  almost all numbers are normal.  The result is arguably the most elementary
  proof for this theorem so far and should be accessible to undergraduates in
  their first year.

  \textbf{Mathematics Subject Classification (2020):} 11K16

  \textbf{Keywords:} normal numbers
\end{abstract}

\section{Introduction}

In 1909, \'{E}mile Borel \cite{bor} introduced \textit{normal numbers} and proved
that almost all numbers are normal.\footnote{According to \cite{doo}, the
  original proof was "unmendably faulty," though.}  Today, several different
proofs of Borel's theorem exist.  This includes some that are generally
considered to be elementary, e.g., \cite{kac}, \cite{nil}, and \cite{fil}.

One of the earliest proofs can be found in Felix Hausdorff's
\textit{Grund\-z\"{u}ge der Mengenlehre} \cite{hau} from 1914.  But Hausdorff only
proved that almost all numbers are \textit{simply normal} in base~2 and then
claimed it would be "evident" that the statement was true for other bases as
well.  He didn't define normal numbers and gave no indication how to
prove a stronger version of his result.  As we will show in this article,
Hausdorff's argument isn't hard to generalize, although the way to do it might
not be totally obvious either.

Anyway, to the author's knowledge, nobody has picked up Hausdorff's elegant
idea so far.  \cite{har} and \cite{niv} contain proofs which argue along
similar lines but require more technical finesse and are less direct.

This article is intended to be accessible to undergraduates at the beginning
of their studies and we thus won't presuppose a lot of previous knowledge
except for basic combinatorics, basic calculus, and a bit of set theory (up to
the definition of \textit{countable}).  Everything else will be defined and
proved, including enough (informal) measure theory to state and prove the main
theorem.

\section{"Almost all"}

The idea of a \textit{measure} is to assign non-negative numbers to sets (of
real numbers) in such a way that these numbers can intuitively be interpreted
as the sizes of the sets.  Two obviously meaningful requirements for a measure
are that the empty set is assigned the measure zero and that the measure is
\textit{additive}: the measure of the union of two (or finitely many) sets
which are mutually disjoint must be the sum of their measures.  In order to be
useful in analysis, measures are actually required to be
\textit{$\sigma$-additive}: the above must also hold for countably many sets
(in which case the sum of the measures becomes a series).

The most important measure, and the one to be used in this article, is the
\textit{Lebesgue measure} which we'll denote with the letter~$\lambda$.  The
Lebesgue measure of an interval of real numbers is its length, e.g.,
$\lambda([1,5])=5-1=4$.  And the open interval $(1,5)$ has the same measure.
This implies that the endpoints "don't count:" finite sets like $\{1,5\}$ are
\textit{null sets}, their measure is zero.  Generally, a set is a null set if,
for every $\epsilon>0$, one can find countably many intervals such that the
set is a subset of the union of these intervals and the sum of the measures of
the intervals is at most~$\epsilon$.  An important example of a null set is
the set~$\Q$ of rational numbers.  More generally:

\begin{lemma}
  Every countable set is a null set.
\end{lemma}
\begin{proof}
  Let $A$ be a countable set and let $(a_n)_{n\in\N}$ be an enumeration
  of~$A$.\footnote{The enumeration doesn't have to be one-to-one, so the proof
    also applies to finite sets.  Also note that in this article $0\in\N$ and
    we write $\N^+$ for the set of \textit{positive} integers.}  Furthermore,
  let $\epsilon$ be an arbitrary positive number.  For each $n\in\N$, let
  $I_n$ be an interval of length $\epsilon/2^{n+1}$ which includes $a_n$.  $A$
  is then covered by the intervals $I_n$ and we have:
  \[ \sum_{n=0}^\infty \lambda(I_n) = \sum_{n=0}^\infty \frac{\epsilon}{2^{n+1}} = \epsilon \qedhere \]
\end{proof}

\begin{lemma}\label{lem-count}
  The countable union of null sets is again a null set.
\end{lemma}
\begin{proof}
  The main idea is that, for a given~$\epsilon$, we'll cover the first null
  set with intervals which have a \textit{total} measure of at most $\epsilon/2$, the
  second one with intervals with a total measure of at most $\epsilon/4$, and
  so on, i.e., we'll again use the geometric series.  The details are left to
  the reader.
\end{proof}

We also note in passing that a subset of a null set is a null set and that
more generally the additivity of~$\lambda$ implies $\lambda(A)\leq\lambda(B)$
for $A\subseteq B$ because $B$ is the disjoint union of $A$ und
$B\setminus A$ and thus $\lambda(B)=\lambda(A)+\lambda(B\setminus A)$.

If $A$ is a set with a positive measure and $P$ is a property the elements
of~$A$ can have or not have, then we say that \textit{almost all} elements
of~$A$ have property~$P$ if the set of numbers \textit{not} having
property~$P$ is a null set.

\section{Normal numbers}

For $\alpha\in\R$ let
\[ \alpha = \sum_{j=-d}^\infty \alpha_{j,r} r^{-j} \] be its representation in
\textit{base} (or \textit{radix})~$r\in\N\setminus\{0,1\}$.  This means that
all $\alpha_{j,r}$ are \textit{digits} in this base, i.e., elements of the set
$\Sigma_r=\{0,\dots,r-1\}$.  In order to ensure uniqueness, we also require:
\begin{enumerate}[label=(\roman*)]
\item $\alpha_{-d,r}\neq0$ for $\alpha\neq0$ and $d=-1$ for $\alpha=0$
\item There is no $k$ with $\alpha_{j,r}=r-1$ for all $j\geq k$.
\end{enumerate}
We sometimes write $\alpha_j$ instead of $\alpha_{j,r}$ if $r$ is implicit.

For example, for $\alpha=\pi$ and $r=10$ we have $d=0$, $\alpha_{0}=3$,
$\alpha_{1}=1$, $\alpha_{2}=4$, and so on.  For $\beta=1/2$ and $r=2$ we have
$d=-1$, $\beta_1=1$ and $\beta_j=0$ for $j>1$.  Note that the alternative
representation $d=-2$ and $\beta_j=1$ for all~$j$ is forbidden by the second
requirement above.\footnote{We will soon see that this decision is
  irrelevant in the context of normality.}

To make up for the fact that we don't have enough digits for bases greater
than~10, we will sometimes write $[a]$ for the $a$-th digit.  For example,
$\pi$ in base $r=100$ will start like this:
\[ \pi = [3].[14][15][92]\dots \] This means that $\alpha_2$ is $[15]$, the
15th digit in base~100.\footnote{The 15th digit is usually written as
  \texttt{F} in hexadecimal notation.}

We now define for $b\in\{0,\dots,r\}$:
\begin{equation}
p_{b,n,\alpha,r} = |\{j \in \N^+ : j \leq n \land \alpha_{j,r} = b \}|\label{eq-p}
\end{equation}
This counts how often the digit~$b$ occurs among the first $n$ digits after
the radix point.\footnote{We will usually only be concerned with digits
  \textit{after} the radix point and will from now on in most cases omit the
  phrase "after the radix point."}  Again, we might omit the $r$ (and
even the $\alpha$).

As $\alpha=\pi$ starts like this
\[ 3.14159265\bm358979\bm32\bm3846264\bm3\bm38\bm3279502884197169\bm399\bm37510\dots \]
in base $r=10$, we have $p_{3,50,\pi}=8$.

A number $\alpha \in \R$ is called \textit{simply normal in base~$r$}, if
\begin{equation}\label{eq-lim}
  \lim_{n\to\infty} p_{b,n,\alpha,r} / n = 1 / r
\end{equation}
holds for all digits $b\in\Sigma_r$, i.e., if each digit occurs with the same
relative frequency "in the long run."

An example that demonstrates how this property depends on the base is the
number $1/3$ which is simply normal in base~2, but obviously \textit{not}
simply normal in the usual base~10.

$\alpha$ is called \textit{normal in base~$r$} if $r^m\alpha$ is simply normal
in base~$r^n$ for all $n\in\N^+$ and all $m\in\N$.  So, the number
$\alpha=1/3$ from above is simply normal in base~$r=2$, but it is \textit{not}
normal in this base because it is not simply normal in base~4 (using $m=0$ and
$n=2$): in base~4 we have $\alpha_j=1$ for all $j\in\N^+$ and thus the limit
in~\eqref{eq-lim} is~1 for the digit~$b=1$ and 0 for the other three digits, but
never~$1/4$.

\begin{lemma}
  If $r^m\alpha$ is simply normal in base~$r^n$ for all $n\in\N^+$ and all
  $m<n$, then $\alpha$ is normal in base~$r$.
\end{lemma}
\begin{proof}
  Let $n$ be fixed and use $r^n$ as the base.  For $m\geq n$ we find non-negative integers $c$ and~$k$ 
  with $c<n$ and $m=kn+c$, and we have:
  \[ r^m\alpha=r^c\cdot(r^n)^k\alpha \] The digits after the radix point of
  $(r^n)^k\alpha$ are the same as those of $\alpha$ beginning at the $k$-th
  digit.  And thus the digits after the radix point of~$r^m\alpha$ are the
  same as those of~$r^c\alpha$ beginning at the $k$-the digit.

  The reason for this is that the digits in base~$r^n$ are obtained by
  combining groups of $n$ digits in base~$r$.
  Multiplication with $r^c$ thus has the same effect on both sequences of
  digits.
\end{proof}

The following example demonstrates the "shift effect" described in the proof
above (for a number~$\alpha$ which is obviously \textit{not} simply normal in
base~10 or 1000):\footnote{The horizontal line (\textit{vinculum}) marks an
  infinitely repeating digit sequence.  The subscript denotes the base.}
\begin{align*}
  r&=10 \\
  n&=3 \\
  m&=7 = 2\cdot n+1\\
  \alpha&=0.[123]\overline{[345][42]}_{1000} = 0.123\overline{345042}_{10} \\
  r^m\alpha&= 10^7\cdot 0.123\overline{345042}_{10} = 1233450.\overline{423450}_{10} \\
  &= [1][233][450].\overline{[423][450]}_{1000} \\
  r^c\alpha&= 10\cdot 0.123\overline{345042}_{10} = 1.23\overline{345042}_{10} = [1].[233]\overline{[450][423]}_{1000}
\end{align*}

We will call a finite sequence $w=b_1b_2b_3\dots b_n$ of digits in
base~$r$ an \textit{$r$-word} (or simply a \textit{word}) and write $|w|$ for
its length~$n$.  For the word $\alpha_{n_1,r}\dots\alpha_{n_2,r}$ consisting
of the digits of~$\alpha$ beginning at position $n_1$ and ending at
position~$n_2$, we'll write $\alpha_{[n_1,n_2],r}$.  For an arbitrary word~$w$
we define:
\[ p_{w,n,\alpha,r} = |\{j \in \N^+ : j +|w|-1\leq n \land
  \alpha_{[j,j+|w|-1],r} = w \}| \] This number counts how often the block~$w$
of digits appears as a substring of~$\alpha_{[1,n],r}$, i.e., of the first $n$
digits of~$\alpha$.  Note that this definition agrees with~\eqref{eq-p} for
words consisting of just one digit.

As an example, consider $\alpha=0.11010111011$.  We have $p_{101,11,\alpha}=3$
which means that the word 101 occurs three times among the first 11 digits.
Note that it doesn't matter that the first two occurrences overlap.

\begin{lemma}\label{lem-freq}
  If $\alpha$ is normal in base~$r$, then we have for all $r$-words~$w$:
  \[ \lim_{n\to\infty} p_{w,n,\alpha,r} / n = 1 / r^{|w|} \]
\end{lemma}

So, if $\alpha$ is normal in some base, then each finite sequence of digits of
this base, no matter how long, appears infinitely often in the representation
of~$\alpha$ and with the same frequency "in the long run" as all other
sequences of the same length.  Which is pretty fascinating if you think about
it.  Imagine the text of your favorite book stored in a computer file and
viewed as a sequence of ones and zeros.  If $\alpha$ is normal in base~2, then
your book will appear infinitely often in the binary representation
of~$\alpha$, as will any other book---and your favorite songs as well!

One should think that such "magic" numbers are pretty rare or don't exist at
all.  But the whole purpose of this article is to prove that they are "normal"
in the sense that numbers which \textit{don't} have this strange property are
extremely scarce.  On the other hand, we don't \textit{know} many normal
numbers.  The ones we do know about were "breeded" for this purpose while the
numbers we deal with on a daily basis are either obviously not normal, like
the rational numbers, or it is unknown whether they are normal.  It is for
example an open question whether $\sqrt2$, $\pi$, or $\mathrm e$ are normal.

By the way, the property described in lemma~\ref{lem-freq} is sometimes used
to \textit{define} normality.  And it is in fact equivalent to our definition.
However, proving the equivalence requires a lot of technical effort which
we'll forego.  See \cite{niv} if you're interested.

Proving lemma~\ref{lem-freq} is not that hard, though.  But instead of a
formal proof (which would probably be confusing because of the notation),
we'll go through an example which is hopefully illuminating enough to
illustrate the general idea.\footnote{If you really need a formal proof, grab
  any normal number.  A very elegant proof, much better than anything I could
  write, can be found somewhere in its digit sequence\dots} Consider the base
$r=2$, a number~$\alpha$ normal in this base, and the word $w=11$ consisting
of two digits.  Furthermore, let $\epsilon$ be some positive real number.
Because $\alpha$ is normal in base~2, it is simply normal in base~4.  That
means we can find a number $n_1$ such that for $n\geq n_1$ approximately $n/4$
of the first $n$ digits in base~4 are the digit~3.  \textit{Approximately}
here is supposed to mean that the actual number deviates from $n/4$ by no more than
$\epsilon n$.  But that implies that among the first $2n$ digits in the base~2
representation of~$\alpha$ we will have $n/4\pm\epsilon n$ occurrences of the
word~11 (which corresponds to the digit~3 in base~4).
\begin{align*}
  \text{Base 4:} & \quad\texttt{ 0 2 1 0 0 \underline{3} 2 \underline{3} 1 \underline{3} \underline{3} 2 0 1 0 0 1 2 2 \underline{3} 0 1 \underline{3} 2 0} \\
  \text{Base 2:} & \quad\texttt{0010010000\underline{11}10\underline{11}01\underline{11}\underline{11}1000010000011010\underline{11}0001\underline{11}1000}
\end{align*}

And because $\alpha$ is normal in base~2, $2\alpha$ is also simply normal in
base~4.  Which entails that we can find a number~$n_2$ which has the same
property for~$2\alpha$ that $n_1$ has for~$\alpha$.  And we can certainly
arrange for $n_2$ to be at least as big as~$n_1$.  So, for $n\geq n_2$ we'll
again find $n/4\pm\epsilon n$ occurrences of the word~11, this time among the
first $2n$ digits of the base~2 representation of~$2\alpha$.  But that's just
the base~2 representation of~$\alpha$ shifted by one digit and so these
are \textit{new} occurrences we haven't counted yet.
\begin{align*}
  \text{Base 4:} & \quad\texttt{ 1 0 2 0 1 \underline{3} 1 2 \underline{3} \underline{3} \underline{3} 0 0 2 0 0 \underline{3} 1 1 2 0 \underline{3} \underline{3} 0 1} \\
  \text{Base 2:} & \quad\texttt{0100100001\underline{11}0110\underline{11}\underline{11}\underline{11}0000100000\underline{11}01011000\underline{11}\underline{11}0001}
\end{align*}

Combined with the ones we already had we now have $n/2\pm\epsilon 2n$ places
where 11 is a substring.  Apart from the possible deviation by~$\epsilon 2n$
that's one quarter of $2n$ base~2 digits and that's what we needed to show.

The final definition is the following: $\alpha$ is called \textit{(absolutely)
  normal} if it is normal in any integer base greater than~1.

\section{The main lemma}

The proof that almost all numbers are normal relies on a technical lemma which
generalizes a computation from~\cite[p.~420\,f]{hau}:

\begin{lemma}\label{lem-main}
  If $r$ is an integer greater than~1, then there's a positive constant~$D$ such that
  the following inequality holds for all $n\in\N^+$:
\[ \sum_{p=0}^n\binom np \frac{(r-1)^{n-p}}{r^{n}} \biggl(\frac pn-\frac 1r\biggr)^{\!4} \leq \frac{D}{n^2} \]  
\end{lemma}
\begin{proof}
We fix positive integers~$s$ and~$n$ and define some functions recursively:
\begin{align}
  f_0(x,y) &= \sum_{p=0}^n\binom np x^{sp} y^{n-p} \notag \\
  f_{k+1}(x,y) &= x\cdot\frac{\partial}{\partial x}f_k(x,y) - y\cdot\frac{\partial}{\partial y}f_k(x,y) \quad\quad\quad(k\in\N) \notag \\
\intertext{By working with individual summands, it is easy to check that}
  f_k(x,y) &= \sum_{p=0}^n\binom np \bigl((s+1)p-n\bigr)^k x^{sp}y^{n-p} \label{binom}
\end{align}
holds for all~$k$.

By the binomial theorem, we know that $f_0$ can also be written like this:
\[ f_0(x,y) = (x^s+y)^n \]
It is a tedious---but completely elementary---exercise to compute $f_4$ based on this representation.\footnote{As we don't live in Hausdorff's times anymore, we can use a computer algebra system.}  We get:
\begin{align*}
  f_4(x,y) &= n(x^s+y)^{n-4} Q \\
  \text{with } Q&= 6n^2(s+1)^2x^sy(sx^s-y)^2 + n^3(sx^s-y)^4 + {} \\
  &\mathrel{\phantom{=}} (s+1)^4x^sy(x^{2s}-4x^sy+y^2) + {} \\
  &\mathrel{\phantom{=}} n(s+1)^3x^sy(7(s+1)x^sy-4sx^{2s}-4y^2)
\end{align*}

We now set $r=s+1$, $x=1/\sqrt[s]r$, and $y=(r-1)/r$.
This implies $sx^s - y = 0$ and two of the four summands in~$Q$ vanish.
The remaining terms simplify to this:
\begin{align*}
  f_4(x,y) &= 3(r-1)^2n^2 + (r^3-7r^2+12r-6)n \\
\intertext{That's a second degree polynomial in~$n$ and we thus know that}
  f_4(x,y) &\leq Cn^2
\end{align*}
for some constant~$C$ independent of~$n$.

If we now replace $f_4(x,y)$ with the term from~\eqref{binom}, we get:
\[ \sum_{p=0}^n\binom np \frac{(r-1)^{n-p}}{r^{n}} (rp-n)^4 \leq Cn^2 \]
Dividing by~$(rn)^4$ yields the inequality we're after with $D=C/r^4$.
\end{proof}

\section{Almost all numbers are normal.}

For the rest of this text, we will concentrate on numbers in the interval
$[0,1)$.  We fix some base $r\geq2$.  If we look at a specific sequence of $n$
digits, then the set of numbers starting with this sequence is an interval
with Lebesgue measure~$1/r^n$.  For example, in base~10, the set of numbers
starting with the sequence~$141$ is the interval $[0.141,0.142)$ which
includes numbers like $\pi-3$.  Its measure is $1/1000$.

We now also fix a specific digit~$b$ of $\Sigma_r$.  We want to know the
measure of the set of numbers that have exactly $p$ occurrences of this digit
among their first $n$ digits.  That's also easy to compute: There are
$\binom np$ ways to pick $p$ of the available~$n$ positions.  For the
remaining $n-p$ positions we can pick any of the other $r-1$ digits and there
are $(r-1)^{n-p}$ ways to do that.  And each of the sequences thus created
results in an interval of length $1/r^n$ disjoint from all other intervals of
the same type.  The measure therefore is:
\begin{equation}\label{eq-meas1}
\binom np \cdot (r-1)^{n-p} \cdot \frac{1}{r^n}
\end{equation}

For a positive real number~$\epsilon$ we now look at the set of all numbers
where the relative frequency of $b$'s among the first $n$ digits deviates from
the "expected" value $1/r$ by at least~$\epsilon$:
\begin{align}
  M_b(n,\epsilon) &= \{\alpha \in [0,1) : |p_{b,n,\alpha,r}/n - 1/r| \geq \epsilon \} \notag \\
\intertext{With~\eqref{eq-meas1}, we can compute the measure of this set:}
  \lambda(M_b(n,\epsilon)) &= \sum_{\substack{p=0\\|p/n-1/r|\geq\epsilon}}^n \binom np \frac{(r-1)^{n-p}}{r^n} \notag \\
\intertext{Using the constant $D$ from lemma~\ref{lem-main} we get}
  \epsilon^4 \cdot \lambda(M_b(n,\epsilon)) &=  \sum_{\substack{p=0\\|p/n-1/r|\geq\epsilon}}^n\binom np \frac{(r-1)^{n-p}}{r^{n}} \cdot \epsilon^4 \notag \\
  &\leq \sum_{\substack{p=0\\|p/n-1/r|\geq\epsilon}}^n\binom np \frac{(r-1)^{n-p}}{r^{n}} \biggl(\frac pn-\frac 1r\biggr)^4 \leq \frac{D}{n^2} \notag \\
\intertext{and thus:}
  \lambda(M_b(n,\epsilon)) &\leq \frac{D}{\epsilon^4}\cdot \frac1{n^2} \label{eq-const}
\end{align}

Let $M_b(\epsilon)$ be the set of numbers $\alpha\in[0,1)$ where the relative
frequency $p_{b,n,\alpha,r}/n$ deviates from~$1/r$ by at least~$\epsilon$ \textit{for
infinitely many}~$n$.  In other words, $\alpha\in M_b(\epsilon)$ iff for each
$m\in\N$ there's an $n\geq m$ such that $\alpha\in M_b(n,\epsilon)$:
\begin{align*}
  M_b(\epsilon) &= \bigcap_{m=1}^\infty S_b(m,\epsilon) \\
  S_b(m,\epsilon) &= \bigcup_{n=m}^\infty M_b(n,\epsilon) 
\end{align*}

By~\eqref{eq-const}, we have
\[ \lambda(S_b(m,\epsilon)) \leq \sum_{n=m}^\infty \lambda(M_b(n,\epsilon))
  \leq \frac{D}{\epsilon^4} \sum_{n=m}^\infty \frac1{n^2} \] and because the
series on the right converges, the measure of $S_b(m,\epsilon)$ will become
arbitrarily small if $m$ is just big enough.  $M_b(\epsilon)$ must therefore
be a null set as it is contained in all $S_b(m,\epsilon)$.

Finally, let $M_b$ be the set of all numbers in $[0,1)$ that are not simply
normal in base~$r$ because condition~\eqref{eq-lim} is violated by at least
the digit~$b$.  By the definition of a limit, $M_b$ will look like this
\[ M_b = \bigcup_{k=1}^\infty M_b(1/k) \] and as a countable union of null
sets it is itself a null set by lemma~\ref{lem-count}.  The set of the
elements of~$[0,1)$ which are not simply normal in base~$r$ is then also a
null set as it is the union of the $r$ sets $M_0$ to~$M_{r-1}$.  We just proved:

\begin{theorem}
  If $r\geq2$ is an arbitrary base, then almost all numbers are simply normal
  in this base.\footnote{We can drop the restriction to the interval $[0,1)$
    if we want.  The proof obviously works just as well for any interval
    $[m,m+1)$ where $m$ is an integer and $\R$ is the countable union of such
    intervals.}
\end{theorem}

We are not quite done yet, but the rest is fairly easy.  Let's again fix a
base~$r$.  If we multiply each element of~$[0,1)$ with a factor~$r^m$ for some
$m>0$, then the set of products is "spread" over the following intervals:
\[ r^m\cdot [0,1) = [0,1) \cup [1,2) \cup \dots \cup [r^m-1,r^m) \] But as the
set of numbers not simply normal in base~$r$ in each of these intervals is a
null set (we just proved that), their union is also a null set, again by
lemma~\ref{lem-count}.

Another application of lemma~\ref{lem-count} yields that the set of numbers
$\alpha\in[0,1)$ such that $r^m\alpha$ is not simply normal in base~$r$ for at
least one~$m\in\N$ is a null set.  But the same argument also works for the
bases $r^2$, $r^3$, and so on.  Invoking lemma~\ref{lem-count} a third time we
get:

\begin{corollary}
  If $r\geq2$ is an arbitrary base, then almost all numbers are normal in this
  base.
\end{corollary}

Finally, you guessed it, we use lemma~\ref{lem-count} for the last time,
utilizing that there are only countably many bases:

\begin{corollary}
  Almost all numbers are absolutely normal.
\end{corollary}

\begin{small}

\end{small}


\begin{thebibliography}{9}
\bibitem{bor}
Borel, \'{E}mile.
\newblock Les probabilit\'{e}s d\'{e}nombrables et leurs applications arithm\'{e}tiques.
\newblock \textit{Suppl.\ Rend.\ Circ.\ mat.\ Palermo} \textbf{27}, 247--271 (1909)

\bibitem{doo}
Doob, Joseph Leo.
\newblock The development of rigor in mathematical probability.
\newblock \textit{Am.\ Math.\ Monthly} \textbf{103}(4), 586--595 (1996)

\bibitem{fil}
Filip, Ferdin\'{a}nd and \v{S}ustek, Jan.
\newblock An elementary proof that almost all real numbers are normal.
\newblock \textit{Acta Univ.\ Sapientiae, Math.} \textbf{2}(1), 99--110 (2010)

\bibitem{har}
Hardy, Godfrey Harold and Wright, Edward Maitland.
\newblock \textit{An introduction to the theory of numbers.}
\newblock Clarendon Press, Oxford (1938)

\bibitem{hau}
Hausdorff, Felix.
\newblock \textit{Grundz\"{u}ge der Mengenlehre.}
\newblock Verlag von Veit \& Comp., Leipzig (1914)

\bibitem{kac}
Kac, Mark.
\newblock \textit{Statistical independence in probability, analysis and number theory.}
\newblock Carus Math.\ Monogr., no.\ 12, Wiley, New York (1959)

\bibitem{nil}
Nillsen, Rodney.
\newblock Normal numbers without measure theory.
\newblock \textit{Am.\ Math.\ Monthly} \textbf{107}(7), 639--644 (2000)

\bibitem{niv}
Niven, Ivan.
\newblock \textit{Irrational numbers.}
\newblock Carus Math.\ Monogr., no.\ 11, Wiley, New York (1956)
\end{thebibliography}
\end{document}